\documentclass[a4paper,12pt]{amsart}
\usepackage[width=0.8\paperwidth,height=0.85\paperheight,twoside=false,includehead,includefoot]{geometry}
\usepackage{amsmath,amssymb,amsfonts}
\usepackage{amsthm}

\theoremstyle{plain}
\newtheorem{thm}{Theorem}[section]

\newtheorem{lem}[thm]{Lemma}

\def\N{\mathbf {N}} \def\Z{\mathbf {Z}} \def\Q{\mathbf {Q}}
\def\R{\mathbf {R}} 
 \def\k{\textbf {k}}  
  
\def\z{\zeta}  

\def\zs{{\zeta}^{\star}}
\def\ZZ{\mathcal {Z}}
\font\fivecy=wncyr5 \def\sa{\hbox{\fivecy X}}
 
\def\zf{\zeta_{\mathcal{F}}} 
\def\zfsa{{\zeta}^{\star , \ast}_{\mathcal{F}}}

\begin{document}
\title[Finite Real Multiple Zeta Values]{A Note on Finite Real Multiple Zeta Values}
\author{Hideki Murahara}
%\date{2015.10.7}
\address{Graduate School of Mathematics, Kyushu University \\
744 Motooka Fukuoka-city, Fukuoka, 819-0395 Japan,}
\email{h-murahara@math.kyushu-u.ac.jp} 

\renewcommand{\subjclassname}{\textup{2000} Mathematics Subject Classification}

\begin{abstract}
We prove three theorems on finite real multiple zeta values: the symmetric formula, the sum formula and the height-one duality theorem.  
These  are analogues of their counterparts on finite multiple zeta values.  
\end{abstract}
\maketitle

\section{Main theorems}
For positive integers $k_1, k_2, \ldots , k_{n}$ with $k_1 \geq 2$, the multiple zeta value and the multiple zeta star value (MZV and MZSV, for short) are defined by 
\begin{align*}
\z (k_1, k_2, \ldots, k_n) := \sum_{m_1>m_2>\cdots >m_n \geq 1} \frac {1}{m_1^{k_1}m_2^{k_2}\cdots m_n^{k_n}}, \\
\zs (k_1, k_2, \ldots, k_n) := \sum_{m_1 \geq m_2 \geq \cdots \geq m_n \geq 1} \frac {1}{m_1^{k_1}m_2^{k_2}\cdots m_n^{k_n}}.
\end{align*}

The finite real multiple zeta values (or symmetric multiple zeta values),  which were first introduced by Kaneko and Zagier \cite{kaneko_zagier-2014}, are defined for any positive integers $k_1, k_2, \ldots , k_{n}$ as follows:  
\begin{align*}
\zf ^{\ast} (k_1, k_2, \ldots, k_n) &:= \sum_{i=0}^{n}(-1)^{k_1+k_2+\cdots +k_i}\z^{\ast}(k_i, k_{i-1}, \ldots , k_1) \z ^{\ast}(k_{i+1}, k_{i+2}, \ldots , k_n), \\
\zf ^{\sa} (k_1, k_2, \ldots, k_n) &:= \sum_{i=0}^{n}(-1)^{k_1+k_2+\cdots +k_i}\z^{\sa}(k_i, k_{i-1}, \ldots , k_1) \z ^{\sa}(k_{i+1}, k_{i+2}, \ldots , k_n). 
\end{align*}
Here, the symbols $\z ^{\ast}$ and $\z ^{\sa}$ on the right-hand sides stand for the regularized values coming from harmonic and shuffle regularizations respectively,
i.e.,  real values obtained by taking constant terms of harmonic and shuffle regularizations as explained in   \cite{ihara_kaneko_zagier-2006}. 
In the sums, we understand $\zeta^{\ast}(\emptyset )=\zeta^{\sa}(\emptyset)=1$.

Let $\ZZ$ be the $\Q$-vector subspace of $\R$ spaned by the MZVs. It is known that this is a $\Q$-algebra.
In \cite{kaneko_zagier-2014}, Kaneko and Zagier proved that the difference $\zf^{\ast} (k_1, k_2, \ldots , k_n) -\zf^{\sa} (k_1, k_2, \ldots , k_n)$
is in the principal ideal of $\ZZ$ generated by $\zeta(2)$ (or $\pi^2$), in other words, that the congruence   
\[  \zf^{\ast} (k_1, k_2, \ldots , k_n) \equiv \zf^{\sa} (k_1, k_2, \ldots , k_n)  \pmod{\zeta (2)} \]
holds in $\ZZ$. 
They then defined the finite real multiple zeta value (FRMZV)  $\zf (k_1, k_2, \ldots , k_n)$ as an element 
in the quotient ring $\ZZ/\zeta (2)$ by 
\[  \zf (k_1, k_2, \ldots , k_n) := \zf^{\ast} (k_1, k_2, \ldots , k_n) \bmod \zeta(2). \]
We also refer to the values $\zf^{\ast} (k_1, k_2, \ldots , k_n)$ and $\zf^{\sa} (k_1, k_2, \ldots , k_n)$ as
(harmonic and shuffle versions of) finite real multiple zeta values.\\

In this paper, we prove the following theorems: 
\begin{thm}[Symmetric formula] \label{Symmetric formula}
Let $(k_1, k_2, \ldots , k_n)$ be any index set $(k_i \in \N)$ and let $S_n$ be the symmetric group of degree $n$.
Then, we have
\begin{align*}
\sum_{\sigma \in S_n} \zf (k_{\sigma(1)}, k_{\sigma(2)}, \ldots , k_{\sigma(n)}) = 0    \quad (\text{in }\ZZ/\zeta (2)).
\end{align*}
\end{thm}

\begin{thm}[Sum formula] \label{Sum formula}
Let $(k_1, k_2, \ldots , k_n)$ be any index set $(k_i \in \N)$. 
For positive integers $k, n$ and $i$ with $1 \leq i \leq n \leq k-1$, we have
\begin{align*}
\sum_{ \substack{ k_1 + k_2 + \cdots + k_n = k \\ k_i \geq 2}} \zf^{\ast} (k_1, k_2, \ldots , k_n) \equiv (-1)^{i-1}\biggl( \binom{k-1}{i-1}+(-1)^{n}\binom{k-1}{n-i} \biggr) \z(k),
\end{align*}
where the congruences are $\bmod\, \zeta(2)$ in the $\Q$-algebra $\ZZ$.
\end{thm}

\begin{thm}[Height-one duality theorem] \label{height_one}
For positive integers $k$ and $n$, we have the equality
\begin{align*}
\zf (k, \underbrace{1,\ldots ,1}_{n-1}) = \zf (n, \underbrace{1,\ldots ,1}_{k-1})
\end{align*}
in $\ZZ/\zeta (2)$.
\end{thm}

\section{Proofs}
\subsection{Proof of Theorem \ref{Symmetric formula}}
Let $\k_1,\k_2$ and $\k$ be any index sets. 
We note that the FRMZVs $\zf^{\ast} (k_1, k_2, \ldots , k_n)$ satisfy the harmonic product rule: 
\[ \zf^{\ast}(\k_1)\zf^{\ast}(\k_2) = \zf^{\ast}(\k_1\ast \k_2),  \]
where the right-hand-side is a linear combination of $\zf^{\ast}(\k)'s$ coming from the harmonic product in \cite{hoffman-97}, e.g., $\zf^{\ast}((2)\ast (2))=2\zf^{\ast}(2,2)+\zf^{\ast}(4)$.  

Hoffman's theorem \cite[Theorem 4.1]{hoffman-2008} states that any symmetric sum 
\[ \sum_{\sigma \in S_n} \zf^{\ast}(k_{\sigma(1)}, k_{\sigma(2)}, \ldots , k_{\sigma(n)}) \]
is a polynomial in the Riemann zeta values $\z(k)$. 
His proof only uses the harmonic product rule of MZVs, and hence applies to our $\zf^{\ast}(\k)'s$. 
Therefore, we conclude completely in the similar manner as in \cite{hoffman-2008} that the symmetric sum above is a sum of products of $\zf^{\ast} (k)=(1+(-1)^k)\z (k)$, which is 0 when $k$ is odd and a multiple of $\z(2)$ when $k$ is even. 

\textbf{Remark.}
One can also prove Theorem \ref{Symmetric formula} directly by using the definition. 
For example, we compute 
\begin{align*}
& \sum_{\sigma \in S_3} \zf^{\ast} (k_{\sigma(1)}, k_{\sigma(2)}, k_{\sigma(3)}) \\ 
& = (1 + (-1)^{k_1})(1 + (-1)^{k_2})(1 + (-1)^{k_3}) \sum_{\sigma \in S_3} \z^{\ast}({k_{\sigma (1)}}, {k_{\sigma (2)}}, {k_{\sigma (3)}}) \\
& + ( (-1)^{k_{1}} + (-1)^{k_{2}}) (1 + (-1)^{k_{3}}) (\z (k_1+k_2, k_3) + \z^{\ast}(k_3, k_1+k_2)) \\
& + ( (-1)^{k_{1}} + (-1)^{k_{3}}) (1 + (-1)^{k_{2}}) (\z (k_1+k_3, k_2) + \z^{\ast}(k_2, k_1+k_3)) \\
& + ( (-1)^{k_{2}} + (-1)^{k_{3}}) (1 + (-1)^{k_{1}}) (\z (k_2+k_3, k_1) + \z^{\ast}(k_1, k_2+k_3)).
\end{align*}

When the weight (= sum of the indices) $k$ is odd, the coefficients $(1 + (-1)^{k_1})(1 + (-1)^{k_2})(1 + (-1)^{k_3})$ and $( (-1)^{k_{1}} + (-1)^{k_{2}}) (1 + (-1)^{k_{3}})$ and etc.\ become 0. 
When $k$ is even, the factor $(1 + (-1)^{k_1})(1 + (-1)^{k_2})(1 + (-1)^{k_3})$ becomes $0$ if at least one of $k_i$ is odd. 
When all $k_i$'s are even, then $\sum_{\sigma \in S_3} \zeta^{\ast}({k_{\sigma (1)}}, {k_{\sigma (2)}}, {k_{\sigma (3)}})=\zeta^{\ast} (k_1)\zeta^{\ast} (k_2)\zeta^{\ast} (k_3)-\zeta (k_1+k_2)\zeta^{\ast} (k_3)-\zeta (k_1+k_3)\zeta^{\ast} (k_2)-\zeta (k_2+k_3)\zeta^{\ast} (k_1)+2\zeta (k_1+k_2+k_3)$ is 0 modulo $\zeta (2)$. 
As for the term $( (-1)^{k_{1}} + (-1)^{k_{2}}) (1 + (-1)^{k_{3}})(\z (k_1+k_2, k_3) + \z^{\ast}(k_3, k_1+k_2))$ etc., 
if we write this as 
$((-1)^{k_{1}}+(-1)^{k_{2}})(1+(-1)^{k_{3}})(\z(k_1+k_2)\z^{\ast}(k_3)-\z(k_1+k_2+k_3))$, 
we see either $((-1)^{k_{1}}+(-1)^{k_{2}})(1+(-1)^{k_{3}})=0$ or $(\z(k_1+k_2)\z^{\ast}(k_3)-\z(k_1+k_2+k_3))$ is a multiple of $\z(2)$. 

\subsection{Proof of Theorem \ref{Sum formula}}
We can prove Theorem \ref{Sum formula} exactly in the same manner as in \cite{saito_wakabayashi-2014}. 
Set
\begin{align*}
S_{k, n, i} := \sum_{ \substack{ k_1 + k_2 + \cdots + k_n = k \\ k_i \geq 2}} \zf^{\ast} (k_1, k_2, \ldots , k_n). 
\end{align*} 
 
We notice that the harmonic version of FRMZVs satisfy the harmonic product rule. 
Thus, $S_{k, n, i}$ enjoy the recursion relation in the following lemma, which can be proved exactly in the same way as in \cite[Proposition 2.2]{saito_wakabayashi-2014}.  
\begin{lem} \label{recurrence_relation}
For positive integers $k$, $n$ and $i$ with $2 \leq i+1 \leq n \leq k-1$, we have
\begin{align*}
(n-i)S_{k, n, i} + i S_{k, n, i+1} + (k-n)S_{k, n-1, i} = 0. 
\end{align*}
\end{lem}

We prove Theorem \ref{Sum formula} by backward induction on $n$. 
To this, we need the initial value. 
\begin{lem} \label{initial_values_for_FRMZVs}
For positive integers $k$ and $i$ with $1 \leq i \leq k-1$, we have
\begin{align*}
S_{k, k-1, i} \equiv (-1)^{i-1}\binom{k}{i}\z (k)  \pmod{\zeta (2)}.
\end{align*}
\end{lem}
\begin{proof}
Since $S_{k, k-1, i} = \zf^{\ast}(\underbrace{1, \ldots ,1}_{i-1}, 2, \underbrace{1, \ldots ,1}_{k-i-1})$, we compute $\zf^{\sa} (\underbrace{1, \ldots ,1}_{i-1}, 2, \underbrace{1, \ldots ,1}_{k-i-1})$ instead. 
Because of the fact $\zf^{\sa} (1, \ldots ,1)=0$, we have by definition
\begin{align*} 
S_{k, k-1, i} &\equiv \zf^{\sa} (\underbrace{1, \ldots ,1}_{i-1}, 2, \underbrace{1, \ldots ,1}_{k-i-1}) \pmod{\zeta (2)} \\
& = \z^{\sa} (\underbrace{1, \ldots ,1}_{i-1}, 2, \underbrace{1, \ldots ,1}_{k-i-1}) + (-1)^k \z^{\sa} (\underbrace{1, \ldots ,1}_{k-i-1}, 2, \underbrace{1, \ldots ,1}_{i-1}).  
\end{align*}  
By using \cite[eq.(5.2)]{ihara_kaneko_zagier-2006} for $w_0=xy^{l}$, we have $\z^{\sa} (\underbrace{1, \ldots ,1}_{m}, 2, \underbrace{1, \ldots ,1}_{l-1})=(-1)^{m}\binom{m+l}{m}\z (2, \underbrace{1, \ldots ,1}_{m+l-1})$. Thus,  
\begin{align*}
S_{k, k-1, i} \equiv (-1)^{i-1} \biggl( \binom{k-1}{i-1}+\binom{k-1}{i} \biggr)\z (2, \underbrace{1, \ldots ,1}_{k-2}) 
 = (-1)^{i-1}\binom{k}{i}\z (k) \pmod{\zeta (2)}.
\end{align*}  
\end{proof} 

Let us consider the case $n=k-1$ of Theorem \ref{Sum formula}. 
If $k$ is even, the identity holds from Lemma \ref{initial_values_for_FRMZVs}. 
If $k$ is odd, then $n$ is even and the identity again follows because
\begin{align*} 
\label{eq7_1}
\binom{k-1}{i-1}+(-1)^n\binom{k-1}{n-i}=\binom{k-1}{i-1}+\binom{k-1}{i}=\binom{k}{i}. 
\end{align*}
We assume the identity holds for $n$. 
By Lemma \ref{recurrence_relation},  
\begin{align*}
(n-k)S_{k, n-1, i} &= (n-i)S_{k, n, i} + i S_{k, n, i+1} \\
& = (n-i)(-1)^{i-1} \biggl( \binom{k-1}{i-1} + (-1)^n \binom{k-1}{n-i} \biggr) \z(k) \\
&\quad + i(-1)^i \biggl( \binom{k-1}{i} + (-1)^n \binom{k-1}{n-i-1} \biggr) \z(k) \\
& = (-1)^{i-1} \biggl( (n-i) \binom{k-1}{i-1} + (k-n+i) (-1)^{n} \binom{k-1}{n-i-1} \biggr) \z(k) \\ 
&\quad + (-1)^i \biggl( (k-i) \binom{k-1}{i-1} + i (-1)^{n} \binom{k-1}{n-i-1} \biggr) \z(k) \\
& = (n-k) (-1)^{i-1} \biggl( \binom{k-1}{i-1} + (-1)^{n-1} \binom{k-1}{n-i-1} \biggr) \z(k). 
\end{align*}
Thus, the identity holds for $n-1$. 

\textbf{Remark.}
We mention an analogy of Theorem \ref{Sum formula} on finite real multiple zeta star values.  
For positive integers $k_1, k_2, \ldots , k_{n}$, let us define $\zfsa$ by  
\[ \zfsa(k_1, k_2, \ldots , k_n) := \sum_{ \substack{\circ \textrm{ is either a comma "," } \\ \textrm{ or a plus "+"}} } \zf^{\ast}(k_1 \circ k_2 \circ \cdots \circ k_n). \]   
Set ${S}^{\star}_{k, n, i} := \sum_{ \substack{ k_1 + k_2 + \cdots + k_n = k \\ k_i \geq 2}} \zfsa (k_1, k_2, \ldots , k_n)$. 
Since these $\zfsa(k_1, k_2, \ldots , k_n)$ satisfy the same harmonic product rule as $\zs (k_1, k_2, \ldots , k_n)$, $S^{\star}_{k, n, i}$ enjoy the same recursion relation as \cite[Proposition 2.2]{saito_wakabayashi-2014}, that is, $ (n-i)S^{\star}_{k, n, i} + i S^{\star}_{k, n, i+1} - (k-n)S^{\star}_{k, n-1, i} = 0$. 
Writing $\k_i \sqcup \k_j$ for juxtaposition of index sets $\k_i$ and $\k_j$, we see from \cite[Theorem 3.1]{hoffman-2008} that 
\[ \zfsa (k_n, k_{n-1}, \ldots , k_1) = (-1)^n\sum_{\k_1 \sqcup \cdots \sqcup \k_l = (k_1, k_2, \ldots , k_n)} (-1)^l \zf^{\ast}(\k_1)\cdots\zf^{\ast}(\k_l). \] 
Consider the case $(k_1, k_2, \ldots , k_n)=(\underbrace{1, \ldots ,1}_{k-i-1}, 2, \underbrace{1, \ldots ,1}_{i-1})$ in this equality. 
Since $\zf^{\ast}(1, \ldots ,1) \equiv \zf^{\sa} (1, \ldots ,1)=0 \pmod{\zeta (2)}$, the right-hand-side is equal modulo $\z(2)$ to $\zf^{\sa}(\underbrace{1, \ldots ,1}_{i-1}, 2, \underbrace{1, \ldots ,1}_{k-i-1})$.  
Thus, we find $S^{\star}_{k, k-1, i} \equiv S_{k, k-1, i} \equiv (-1)^{i-1}\binom{k}{i}\z (k) \pmod{\zeta (2)}$. 
In a similar way as the proof of Theorem \ref{Sum formula} (i.e., by backward induction on $n$), we get
\begin{align*}
\sum_{ \substack{ k_1 + k_2 + \cdots + k_n = k \\ k_i \geq 2}} \zfsa (k_1, k_2, \ldots , k_n) \equiv (-1)^{i-1} \biggl( (-1)^n\binom{k-1}{i-1}+\binom{k-1}{n-i} \biggr) \z (k) \pmod{\zeta (2)}.  
\end{align*}

\subsection{Proof of Theorem \ref{height_one}}
For a given index $\k$, we call the number of its elements greater than $1$ the height. 
With this terminology, we shall call 
\[ \zf (k, \underbrace{1, \cdots ,1}_{n-1}) \]
 height one FRMZVs. 
In this section, we prove Theorem \ref{height_one}. 
To this, we state the following key lemma.
\begin{lem} \label{key_lemma}
For positive integers $k$ and $n$ with $k \geq 2$, we have 
\begin{align*}
\z^{\sa} (\underbrace{1, \ldots ,1}_{n-1}, k) = (-1)^{n-1} \zs (k, \underbrace{1,\ldots ,1}_{n-1}).
\end{align*}
\end{lem}

\begin{proof}
We note that the MZVs $\z^{\sa} (k_1, k_2, \ldots , k_n)$ satisfy the shuffle product rule (for precise definition in \cite{hoffman-97}) coming from the iterated integral expressions of the MZVs: $\z^{\sa}(\k_1)\z^{\sa}(\k_2) = \z^{\sa}(\k_1 \sa \k_2)$, e.g., $\z^{\sa}((1, 1) \sa (2))=3\z^{\sa}(2, 1, 1)+2\z^{\sa}(1, 2, 1)+\z^{\sa}(1, 1, 2)$. 
Here, the notation $\k_1 \sa \k_2$ is a $\Z$-linear combination of indices and we extend $\z^{\sa}$ linearly. 
To make notations easier, let $\z^{\sa} (1\oplus (k_1, k_2, \ldots , k_n))=\z^{\sa}(k_1+1, k_2, \ldots , k_n)$ and $\z^{\sa}(\ldots , l, \underbrace{1,\ldots ,1}_{-1}, m, \ldots)=\z^{\sa}(\ldots , l+m-1, \ldots)$. 
By the regularization formula \cite[eq.(5.2)]{ihara_kaneko_zagier-2006}, we have (extending $1\oplus (\cdot)$ also linearly)
\begin{align*}
\z^{\sa} (\underbrace{1, \ldots ,1}_{n-1}, k)  
& = (-1)^{n-1} \z^{\sa}(1\oplus ((\underbrace{1, \ldots ,1}_{n-1}) \sa (k-1))) \\
& = (-1)^{n-1} \sum_{\substack{{ a_1+\cdots +a_{k}=n-1} \\ a_i \geq 0 \, (i=1, 2, \ldots, k)}} \z(2, \underbrace{1, \ldots ,1}_{a_1-1}, \ldots , 2, \underbrace{1, \ldots ,1}_{a_{k-2}-1}, 2, \underbrace{1, \ldots ,1}_{a_{k-1}+a_k}) \\
& = (-1)^{n-1} \sum_{\substack{{ a_1+\cdots +a_{k-1}=n-1} \\ a_i \geq 0 \, (i=1, 2, \ldots, k-1)}} (a_{k-1}+1) \z(2, \underbrace{1, \ldots ,1}_{a_1-1}, \ldots , 2, \underbrace{1, \ldots ,1}_{a_{k-2}-1}, 2, \underbrace{1, \ldots ,1}_{a_{k-1}}) \\
& = (-1)^{n-1} \sum_{\substack{{ a_1+\cdots +a_{k-1}=n-1} \\ a_i \geq 0 \, (i=1, 2, \ldots, k-1)}} (a_{k-1}+1) \z(a_{k-1}+2, a_{k-2}+1, \ldots , a_1+1). 
\end{align*} 
For the last equality, we used the duality formula of MZVs. 
That the last sum equals $\zs (n, \underbrace{1,\ldots ,1}_{k-1})$ is due to Ohno \cite[Proof of Theorem $2$]{ohno-99}, see also \cite[$\S 3$]{kaneko-2010_2}. 
Thus 
\[ \z ^{\sa} (\underbrace{1, \ldots ,1}_{n-1}, k)=(-1)^{n-1} \zs (k, \underbrace{1,\ldots ,1}_{n-1}).  \]
\end{proof}

Now, we prove Theorem \ref{height_one}.
When either $k$ or $n=1$, the theorem clearly holds. 
We consider the case when $k, n \geq 2$. 
From the above Lemma \ref{key_lemma}, we have
\begin{align*}
\begin{split}
& \zf^{\sa} (k,\underbrace{1,\ldots ,1}_{n-1}) - \zf^{\sa} (n,\underbrace{1,\ldots ,1}_{k-1}) \\
& = \z (k,\underbrace{1,\ldots ,1}_{n-1}) + (-1)^k \zs (k, \underbrace{1,\ldots ,1}_{n-1})-( \z (n, \underbrace{1,\ldots ,1}_{k-1}) + (-1)^n \zs (n, \underbrace{1,\ldots ,1}_{k-1}) ). 
\end{split}
\end{align*}
Let $\psi (X) = \frac{\Gamma '(X)}{\Gamma (X)}$. 
By using the well-known generating series 
\begin{align*}
1 - \sum_{k, n \geq 1} \z (k+1, \underbrace{1, \ldots ,1}_{n-1})X^{k}Y^{n} & = \exp \biggr( \sum_{n \geq 2} \z (n)\frac{X^{n} + Y^{n} - (X +Y)^{n}}{n}\biggl) \\ 
& = \frac{\Gamma(1-X)\Gamma(1-Y)}{\Gamma(1-X-Y)}   
\end{align*} 
(cf. Aomoto \cite{aomoto-1990} and Drinfel'd \cite{drinfel'd-1991}) and $\psi (1-X)=-\sum_{k\geq 2} \z (k)X^{k-1} -\gamma$ ($\gamma$ is Euler's constant.), we have 
\begin{align*}
\begin{split}
& \sum_{k, n \geq 2} \biggr( \z (k, \underbrace{1,\ldots ,1}_{n-1})  - \z (n, \underbrace{1,\ldots ,1}_{k-1}) \biggl) X^{k-1} Y^{n-1}  \\
& = \biggl(\frac{1}{Y}-\frac{1}{X}\biggr) \biggr( 1 - \frac{\Gamma(1-X) \Gamma(1-Y)}{\Gamma(1-X-Y)} \biggl) + \psi (1-X)-\psi (1-Y).  
\end{split}
\end{align*}
On the other hand, from Kaneko and Ohno \cite[Theorem 2]{kaneko_ohno-2009}, 
\begin{align*}
& \sum_{k, n \geq 2} \biggl( (-1)^k \zs (k, \underbrace{1,\ldots ,1}_{n-1}) - (-1)^n \zs (n, \underbrace{1,\ldots ,1}_{k-1}) \biggr) X^{k-1} Y^{n-1}  \\
& = -\psi(X)+\psi(Y)-\pi (\cot(\pi X) - \cot(\pi Y)) \frac{\Gamma(1-X) \Gamma(1-Y)}{\Gamma(1-X-Y)}, 
\end{align*}
From these, and by the well-known equalities: 
\begin{align*}
\pi \cot(\pi X)&=\frac{1}{X}+\psi(1-X)-\psi(1+X), \\
\psi(X) &=\psi(1+X)-\frac{1}{X},  
\end{align*}
we have
\begin{align*}
& \sum_{k, n \geq 2} \biggl( \zf^{\sa}(k, \underbrace{1,\ldots ,1}_{n-1})-\zf^{\sa}(n, \underbrace{1,\ldots ,1}_{k-1}) \biggr) X^{k-1} Y^{n-1} \\ 
& = \biggl(\frac{1}{Y} - \frac{1}{X}\biggr) \biggr( 1-\frac{\Gamma(1-X) \Gamma(1-Y)}{\Gamma(1-X-Y)} \biggl) +\psi (1-X)-\psi (1-Y) \\
& \quad -\psi(X) +\psi(Y) -\pi (\cot(\pi X) - \cot(\pi Y)) \frac{\Gamma(1-X) \Gamma(1-Y)}{\Gamma(1-X-Y)} \\
& = \biggl( 1-\frac{\Gamma(1-X) \Gamma(1-Y)}{\Gamma(1-X-Y)} \biggr) (\psi (1-X)-\psi (1+X)-\psi (1-Y)+\psi (1+Y)) \\
& = -2\biggl( 1-\frac{\Gamma(1-X) \Gamma(1-Y)}{\Gamma(1-X-Y)} \biggr) \sum_{l \geq 1} \z (2l) (X^{2l-1}-Y^{2l-1}).   
\end{align*}
Since the coefficients of $\frac{\Gamma(1-X) \Gamma(1-Y)}{\Gamma(1-X-Y)}$ belong to the $\Q$-algebra $\ZZ$, we have 
\[ \zf^{\sa} (k,\underbrace{1,\ldots ,1}_{n-1}) \equiv \zf^{\sa} (n,\underbrace{1,\ldots ,1}_{k-1})  \pmod{\zeta (2)}. \] 
This proves Theorem \ref{height_one}.

\section*{Acknowledgment}
The author would like to thank Professor Masanobu Kaneko for valuable comments and suggestions.

\end{document}